\documentclass[12pt]{article}
\usepackage{amssymb,amsmath}
\usepackage{amsthm}

\newcommand{\ol}{\overline}

\newcommand{\trdeg}{\mathop{\mathrm{tr.\, deg}}\nolimits}

\newcommand{\gen}[1]{\langle #1 \rangle}

\newcommand{\eq}{\mathrm{Eq}}

\newtheorem{theorem}{Theorem}
\newtheorem{corollary}[theorem]{Corollary}
\newtheorem{lemma}[theorem]{Lemma}
\newtheorem{proposition}[theorem]{Proposition}
\theoremstyle{definition}
\newtheorem{definition}[theorem]{Definition}
\theoremstyle{remark}

\newtheorem*{notation}{Notation}
\newtheorem*{remark}{Remark}

\newenvironment{acknowledgement}{\bigskip\noindent \textit{Acknowledgement.}}{\bigskip}

\title{Differential transcendence of solutions of difference Riccati equations
and Tietze's treatment}

\author{Seiji Nishioka}

\begin{document}

\maketitle

\begin{small}
  \noindent
  {\bf Abstract}\quad
  There is the paper by H. Tietze published in 1905
  on differential transcendence of solutions of difference Riccati equations.
  In this paper, we clarify the essence of Tietze's treatment and make it 
  purely algebraic.
  As an application, the $q$-Airy equation is studied.
  
%  33E30. % Special functions -- Other functions coming from differential, difference and integral equations

\end{small}

\section{Introduction}

These days, several authors study Galois theories on linear difference equations
of higher order than 1, and show results on differentially algebraic character
of solutions especially for equations of second order 
(cf. the papers \cite{DHR,HS2008}).
Historically, there is the paper \cite{Tietze1905} by H. Tietze published in
1905,
which is on differential transcendence of solutions of difference Riccati equations.
The relation between a difference Riccati equation and a second-order
linear difference equation is similar to that of a Riccati equation and
a second-order linear differential equation.
He proposed his normal form 
\[
  \varphi(x+1)=\frac{r(x)}{\varphi(x)}+1
\]
of a difference Riccati equation
\[
  \varphi(x+1)=\frac{A(x)\varphi(x)+B(x)}{C(x)\varphi(x)+D(x)},
\]
and studied differential transcendence of solutions 
when $r(x)$ is a rational function with $\lim_{x\to\infty}r(x)=0$. 
His theorem says that if there is no rational function solution, 
there is no solution satisfying an algebraic differential equation.

In this paper, we clarify the essence of Tietze's treatment and make it 
purely algebraic. As a result, the transform $\varphi(x+1)$ of $\varphi(x)$ 
appeared in the Riccati equation
can be replaced with $\varphi(qx)$, $\varphi(x^2)$, etc.
Here, we call them $\tau(\varphi)$.
Theorem \ref{2017-04-28p1a}, the main theorem, and Proposition \ref{2017-05-25p1a}
say that existence of a differentially algebraic solution is closely related to
existence of a rational function solution $g$ of the third-order linear difference equation
\[
  \begin{aligned}
    &\tau^2(sr)\tau(s)s\tau^3(g)+(\tau(r)+1)\tau(s)s\tau^2(g)\\
    &\quad -(\tau(r)+1)s\tau(g)-rg+s\tau(D(r)/r)=0,
  \end{aligned}
\]
where $(d/dx)\tau(g)=s\tau((d/dx)g)$. For example, $s=1$ when $\tau(g)=g(x+1)$,
$s=q$ when $\tau(g)=g(qx)$, and $s=2x$ when $\tau(g)=g(x^2)$.
In Corollary \ref{2017-07-17p1a}, we shall introduce an application, which 
implies in rough sense that a solution of the Riccati equation 
\[
  y(qt)=-qt+\frac{1}{y}
\]
associated with the $q$-Airy equation
\[
  y(q^2t)+qty(qt)-y(t)=0
\]
has no differentially algebraic relation with solutions of 
linear difference equations when $q$ is not a root of unity.

\begin{notation}
 Throughout the paper every field is of characteristic zero.
 When $K$ is a field and $\tau$ is an isomorphism of $K$ into itself,
 namely an injective endomorphism, the pair $\mathcal{K}=(K,\tau)$ is
 called a difference field. 
 We call $\tau$ the (transforming) operator and $K$ the underlying
 field.
 For a difference field $\mathcal{K}$, $K$ often denotes its underlying field.
 For $a\in K$, the element $\tau^na\in K$ ($n\in \mathbb{Z}$), if it exists,
 is called the $n$-th transform of $a$ and is
 sometimes denoted by $a_n$. 
% We use this abbreviation only in difference equations.
 If $\tau K=K$, we say that $\mathcal{K}$ is inversive.
% If $K/\tau K$ is algebraic, we say that $\mathcal{K}$ is almost inversive.
 %
 For an algebraic closure $\ol{K}$ of $K$, the transforming operator $\tau$
 is extended to an isomorphism $\ol{\tau}$ of $\ol{K}$ into itself,
 not necessarily in a unique way 
 (cf. the book \cite{ZS1958}, Ch. II, \S14, Theorem 33).
 We call the difference field $(\ol{K},\ol{\tau})$ 
 an algebraic closure of $\mathcal{K}$.
 For difference fields $\mathcal{K}=(K,\tau)$ and
 $\mathcal{K}'=(K',\tau')$, $\mathcal{K}'/\mathcal{K}$ is called a
 difference field extension if $K'/K$ is a field extension and
 $\tau'|_K=\tau$. 
 In this case, we say that $\mathcal{K}'$ is a difference overfield of
 $\mathcal{K}$ and that $\mathcal{K}$ is a difference subfield of
 $\mathcal{K}'$.
 For brevity we sometimes use $(K,\tau')$ instead of $(K,\tau'|_K)$. 
 We define a difference intermediate field in the proper way.
 Let $\mathcal{K}$ be a difference field,
 $\mathcal{L}=(L,\tau)$ a difference overfield of $\mathcal{K}$
 and $B$ a subset of $L$.
 The difference subfield $\mathcal{K}\gen{B}_{\mathcal{L}}$ of
 $\mathcal{L}$ is defined to be the difference field
 $(K(B,\tau B,\tau^2 B,\dots),\tau)$ and is denoted by $\mathcal{K}\gen{B}$
 for brevity.
 A solution of a difference equation over $\mathcal{K}$ 
 is defined to be
 an element of some difference overfield of $\mathcal{K}$
 which satisfies the equation (cf. the books \cite{Cohn,Levin2008}).

%  When $R$ is a ring and $\tau$ is an isomorphism of $R$ into itself,
%  the pair $\mathcal{R}=(R,\tau)$ is called a difference ring.
%  Let $\mathcal{R}=(R,\tau)$ and $\mathcal{R}'=(R',\tau')$ be
%  difference rings. A homomorphism $\phi$ of $R$ to $R'$ is called
%  a difference homomorphism of $\mathcal{R}$ to $\mathcal{R}'$ 
%  if $\phi\tau=\tau'\phi$
%  (cf. the books \cite{Cohn,Levin2008}).
  
  For a field $K$, a derivation $D$ of $K$, 
  which is a mapping $D\colon K\to K$ such that $D(a+b)=D(a)+D(b)$
  and $D(ab)=D(a)b+aD(b)$, and a transforming operator $\tau$ 
  of $K$, we simply call the pair $\mathcal{K}=(K,D,\tau)$ a DTC field,
  though some relation between
  the derivation and the transforming operator may be given.
  A DTC overfield $\mathcal{K}'=(K',D',\tau')$ of $\mathcal{K}$ 
  is a differential overfield on $D'$, namely a field extension $K'/K$ with
  $D'|_K=D$, 
  and a difference overfield on $\tau'$.
  Let $\mathcal{L}/\mathcal{K}$ be a DTC field extension and 
  $B$ a subset of $L$.
  The DTC subfield $\mathcal{K}\gen{B}_\mathcal{L}$ of $\mathcal{L}$ is
  defined to be 
  the minimum DTC subfield containing $K$ and $B$.
  
  Let $\mathcal{L}/\mathcal{K}$ be a differential field extension, and 
  $D$ the derivation of $\mathcal{L}$.
  When $f\in L$ satisfies $\trdeg K(f,Df,D^2f,\dots)/K<\infty$,
  we say that $f\in \mathcal{L}$ is differentially algebraic over $K$.
  It is equivalent that $f$ satisfies a certain algebraic differential equation
  over $\mathcal{K}$ in $\mathcal{L}$.
%  Here, $\mathcal{L}/\mathcal{K}$ may be a DTC field extension.
  
%  ...
  
  Let $F/K$ be an algebraic function field of one variable.
  A place $P$ of $F/K$ is the maximal ideal of some valuation ring 
  of $F/K$. The valuation ring and the normalized discrete valuation 
  associated with $P$ is denoted by $\mathcal{O}_P$ and $v_P$,
  respectively.
  A discrete valuation of $F/K$ is a function 
  $v\colon F\to \mathbb{Z}\cup\{\infty\}$
  with the following properties.
  \begin{enumerate}
  \item $v(x)=\infty\iff x=0$.
  \item $v(xy)=v(x)+v(y)$ for all $x,y\in F$.
  \item $v(x+y)\geq\min\{v(x),v(y)\}$ for all $x,y\in F$.
  \item There exists an element $z\in F^\times$ with $v(z)\neq 0$
    ($v(z)=1$ for a normalized discrete valuation).
  \item $v(a)=0$ for all $0\neq a\in K$.
  \end{enumerate}
  For a rational function field $K(x)/K$, $P_\alpha$ ($\alpha\in K$), denotes
  the place which has the prime element $x-\alpha$ and $P_\infty$
  the place which has the prime element $1/x$.
\end{notation}

In Section \ref{normal}, we shall introduce Tietze's normal form of difference 
Riccati equation and the way how to obtain it. 
In Section \ref{lemmas}, we shall introduce several lemmas used in Section \ref{theorem} 
to prove the main theorem.
In Section \ref{application}, an application to $q$-Airy equation is
studied.

\section{Tietze's normal form of difference Riccati equation}\label{normal}

For a second-order linear difference equation,
\[
  y_2+ay_1+by=0,
\]
by setting $z=y_1/y$, we obtain the following first-order difference equation,
\[
  z_1=\frac{-az-b}{z}.
\]
We call this the difference Riccati equation associated with the above equation.
By iterations, we can express $z_i$ in terms of $z$ such as
\[
  z_2=\frac{(a_1a-b_1)z+a_1b}{-az-b}.
\]
Here, we introduce a notation about those iterations.

Let $\mathcal{L}=(L,\tau)$ be a difference field, and let
\begin{gather*}
  A=\begin{pmatrix}a&b\\c&d\end{pmatrix}\in \mathrm{M}_2(L),\\
  A_i=\begin{pmatrix}a^{(i)} & b^{(i)} \\ c^{(i)} & d^{(i)}\end{pmatrix}
  =(\tau^{i-1}A)(\tau^{i-2}A)\cdots(\tau A)A\quad (i=1,2,\dots).
\end{gather*}
In this paper, $\eq(A,i)/\mathcal{L}$ denotes the equation over $\mathcal{L}$,
\[
  y_i(c^{(i)}y+d^{(i)})=a^{(i)}y+b^{(i)}.
\]
We easily see the following.

\begin{lemma}
  If $f$ is a solution of $\eq(A,k)/\mathcal{L}$ in a difference field extension
  $\mathcal{U}/\mathcal{L}$,
  $f\in\mathcal{U}$ is also a solution of $\eq(A,ki)/\mathcal{L}$
  $(i=1,2,\dots)$.
%  This is proved by induction.
\end{lemma}

%\begin{lemma}
%  Let $B=A_k$ and $B_i=(\tau^{k(i-1)}B)(\tau^{k(i-2)}B)\cdots B$
%  $(i=1,2,\dots)$. Then $B_i=A_{ki}$.
%\end{lemma}

%\begin{lemma}
%  For any $k,l,m\in\mathbb{Z}_{>0}$, 
%  \begin{equation*}
%  \begin{aligned}
%    &\text{$f\in\mathcal{L}$ is a solution of Eq$(A_k,lm)/\mathcal{K}^{(k)}$}\\
%    \Longleftrightarrow
%    &\text{$f\in\mathcal{L}^{(l)}$ is a solution of Eq$(A_{kl},m)/\mathcal{K}^{(kl)}$,}
%  \end{aligned}
%  \end{equation*}
%  where $\mathcal{L}$ is a difference overfield of $\mathcal{K}^{(k)}$.
%\end{lemma}

We also define $\eq(f;A;g)$ for $f,g\in L$, which denotes the equality,
\[ f(cg+d)=ag+b. \]

\begin{lemma}\label{2017-04-25pA2}
  When both $\eq(f;A;g)$ and $\eq(g;B;h)$ hold, $\eq(f;AB;h)$ holds.
\end{lemma}

This is obvious.
We shall show another lemma.

\begin{lemma}\label{2017-05-11p1a}
  Let $A\in \mathrm{M}_2(L)$ and $P\in \mathrm{GL}_2(L)$.
  Let $\mathcal{U}=(U,\tau)$ be a difference overfield of $\mathcal{L}$
  and $f\in\mathcal{U}$ a solution of $\eq(A,i)/\mathcal{L}$.
  Define $g\in L(f)$ as the element satisfying $\eq(g;P;f)$
  which uniquely exists.
  Then we find
  \begin{enumerate}
  \item $\eq(f;P^{-1};g)$ holds,
  \item $g\in\mathcal{U}$ is a solution of $\eq((\tau P)AP^{-1},i)/\mathcal{L}$.
  \end{enumerate}
\end{lemma}

\begin{proof}
  {\rm (i)}
  If we use the expression
  \[ P=\begin{pmatrix}a&b\\c&d\end{pmatrix}, \]
  then $P^{-1}$ is expressed as
  \[ P^{-1}=\frac{1}{ad-bc}\begin{pmatrix}d&-b\\-c&a\end{pmatrix}. \]
  Since $\eq(g;P;f)$ denotes
  $g(cf+d)=af+b$,
  the following holds,
  \[ f(-cg+a)=dg-b. \]
  Dividing by $ad-bc$, we find $\eq(f;P^{-1};g)$.
  
  {\rm (ii)} Since $f\in\mathcal{U}$ is a solution of $\eq(A,i)/\mathcal{L}$,
  $f$ satisfies $\eq(\tau^if;A_i;f)$.
  Applying $\tau^i$ to $\eq(g;P;f)$, we obtain
  $\eq(\tau^ig;\tau^iP;\tau^if)$. Hence by Lemma \ref{2017-04-25pA2} 
  we find $\eq(\tau^ig;(\tau^iP)A_iP^{-1};g)$.
  If we let $B=(\tau P)AP^{-1}$, then
  \[
  \begin{aligned}
    B_i&=(\tau^{i-1}B)(\tau^{i-2}B)\cdots B\\
    &=(\tau^iP)(\tau^{i-1}A)(\tau^{i-1}P^{-1})(\tau^{i-1}P)(\tau^{i-2}A)(\tau^{i-2}P^{-1})\cdots(\tau P)AP^{-1}\\
    &=(\tau^iP)A_iP^{-1}.
  \end{aligned}
  \]
  Therefore $g\in\mathcal{U}$ is a solution of $\eq(B,i)/\mathcal{L}$.
\end{proof}

Tietze's normal form of a difference Riccati equation is
\[
  y_1=1+\frac{r}{y}, \quad r\in L^\times.
\]
We shall introduce the way how to transform a difference Riccati equation
into this form.

\begin{proposition}\label{2017-05-11p3a}
  Define 4 forms of matrices 
  $A=\begin{pmatrix}a&b\\c&d\end{pmatrix}\in\mathrm{GL}_2(L)$
  as follows:
  \begin{itemize}
  \item[{\rm (F1)}] $c\neq 0$.
  \item[{\rm (F2)}] $A=\begin{pmatrix}a & b\\ 1 & 0\end{pmatrix}$, $a\neq 0$.
  \item[{\rm (F3)}] $A=\begin{pmatrix}0 & b\\ 1 & 0\end{pmatrix}$, $\tau b\neq b$.
  \item[{\rm (F4)}] $A=\begin{pmatrix}0 & b\\ 1 & 0\end{pmatrix}$, $\tau b=b$.
  \item[{\rm (FT)}] $A=e\begin{pmatrix}1 & r\\ 1 & 0\end{pmatrix}$, where $e,r\in L^\times$.
    This is ``associated'' with Tietze's normal form.
  \end{itemize}
  Then the following hold.
  \begin{enumerate}
  \item For a matrix of {\rm(F1)}, 
    by $P=\begin{pmatrix}c&d\\0&1\end{pmatrix}\in\mathrm{GL}_2(L)$, it becomes
    \[ (\tau P)AP^{-1}=\frac{1}{c}\begin{pmatrix}a\tau(c)+c\tau(d)&-(ad-bc)\tau(c)\\c&0\end{pmatrix},\]
    which is of {\rm(F2)}, {\rm(F3)} or {\rm(F4)}.
  \item For a matrix of {\rm(F2)},
    by $P=\begin{pmatrix}a&b\\a&0\end{pmatrix}\in\mathrm{GL}_2(L)$, it becomes
    \[ (\tau P)AP^{-1}=\tau(a)\begin{pmatrix}1&(1/a)\tau(b/a)\\1&0\end{pmatrix}, \]
    which is of {\rm(FT)}.
  \item For a matrix of {\rm(F3)},
    $P=\begin{pmatrix}1&b\\1&1\end{pmatrix}$ is a regular matrix 
    because of $b\neq 1$. In this case, 
    \[ (\tau P)AP^{-1}=\frac{1}{1-b}\begin{pmatrix}\tau(b)-b&b(1-\tau(b))\\1-b&0\end{pmatrix} \]
    is of {\rm(F2)}.
  \item When a matrix of {\rm(F1)} turns into a matrix of {\rm(FT)} by repeated applications of
    {\rm (i)}, {\rm (ii)}, and/or {\rm (iii)},
    there exist $P_1,\dots,P_n\in\mathrm{GL}_2(L)$ such that
    \[
      (\tau P_n)\cdots(\tau P_1)AP_1^{-1}\cdots P_n^{-1}
      =e\begin{pmatrix}1&r\\1&0\end{pmatrix},
    \]
    and therefore $(\tau Q)AQ^{-1}=e\begin{pmatrix}1&r\\1&0\end{pmatrix}$
    for some $Q\in\mathrm{GL}_2(L)$.
  \item When a matrix of {\rm(F1)} turns into a matrix of {\rm(F4)} 
    by an application of {\rm (i)},
    $B=(\tau P)AP^{-1}=\begin{pmatrix}0&r\\1&0\end{pmatrix}$ $(\tau r=r)$
    satisfies
    \[
      B_2=(\tau B)B=r\begin{pmatrix}1&0\\0&1\end{pmatrix},
    \]
    which is ``associated'' with the difference equation
    $y_2=y$.
  \end{enumerate}
\end{proposition}

The proof is straightforward.

\section{Lemmas}\label{lemmas}

Let $\mathcal{F}=(F,\tau_0)$ be a difference field and
$F/K$ an algebraic function field of one variable. 
Let $A=\begin{pmatrix}1&r\\1&0\end{pmatrix}\in\mathrm{GL}_2(F)$ and
suppose there exists a place $P$ of $F/K$ such that $v_P(\tau_0^ir)>0$ for
all $i\geq 0$.
We also suppose that $\eq(A,i)/\mathcal{F}$ has no solution algebraic over $F$
for any $i\geq 1$.
Let $Y$ be a solution of $\eq(A,1)/\mathcal{F}$, which is transcendental 
over $F$, and let $\mathcal{F}\gen{Y}=(F(Y),\tau)$.

\begin{lemma}\label{2017-04-25p1a}
  If $N\in F[Y]\setminus\{0\}$ satisfies
  \[
    N=\alpha Y^\lambda\cdot Y^\nu\tau(N),
  \]
  where $\alpha\in F^\times$, $Y^\lambda\| N$, $\nu=\deg N$, then
  $N\in F^\times$.
  Here, $Y^\lambda\|N$ means that $Y^\lambda|N$ and $Y^{\lambda+1}\nmid N$.
\end{lemma}

\begin{proof}
  Let $\ol{\mathcal{F}\gen{Y}}=(\ol{F(Y)},\ol{\tau})$ be an algebraic closure
  of $\mathcal{F}\gen{Y}$, and $\ol{\mathcal{F}}$ the algebraic closure of
  $\mathcal{F}$ in it.
  The polynomial $N$ is expressed as
  \begin{equation}\label{2017-04-25p2b}
    N=\beta(Y-1)^{m_0}\prod_{i=1}^l(Y-\mu_i)^{m_i},
  \end{equation}
  where $\beta\in F^\times$, $\mu_i\in\ol{F}$, $\mu_i\neq 1$,
  $\mu_i\neq \mu_j$ $(i\neq j)$, $m_i>0$ $(i\geq 1)$.
  Applying $\tau$, we obtain
  \[
    \tau N=\tau(\beta)\left(\frac{r}{Y}\right)^{m_0}
    \prod_{i=1}^l\left(1+\frac{r}{Y}-\ol{\tau}\mu_i\right)^{m_i},
  \]
  and so 
  \[
    N=\alpha Y^\lambda\cdot Y^\nu\tau(N)
    =\alpha Y^\lambda\tau(\beta)r^{m_0}\prod_{i=1}^l
    ((1-\ol{\tau}\mu_i)Y+r)^{m_i}.
  \]
  From the above two expressions for $N$, we find
  \begin{equation}\label{2017-04-25p2a}
    (Y-1)^{m_0}\prod_{i=1}^l(Y-\mu_i)^{m_i}
    =Y^\lambda\prod_{i=1}^l\left(Y-\frac{r}{\ol{\tau}\mu_i-1}\right)^{m_i}.
  \end{equation}
  Comparing degrees, we obtain $m_0=\lambda$.
  
  Let $H\subset\ol{F}$ be the set of all the roots of $N$.
  For any $\mu\in H\setminus\{0\}$, $1+r/\mu\in\ol{\tau}H$.
  Indeed, from the equation \eqref{2017-04-25p2a},
  $\mu=r/(\ol{\tau}\mu_i-1)$ for some $i$.
  This implies $\ol{\tau}\mu_i=1+r/\mu$, and so $1+r/\mu\in\ol{\tau}H$.
  Define the mapping $\phi\colon H\setminus\{0\}\to H$ by
  $\phi(\mu)=\ol{\tau}^{-1}(1+r/\mu)$, and the sets
  \[
    H_0=\{\mu\in H\;|\;\text{$\phi^i(\mu)=0$ for some $i\geq 0$}\},
    \quad H_\infty=H\setminus H_0.
  \]
  The restriction $\phi|_{H_\infty}$ is a bijection to $H_\infty$, for
  $\phi$ is injective and the number of elements of $H_\infty$ is finite.
  
  To show $H=\emptyset$, firstly, we assume $1\in H_\infty$.
  This implies $\phi^i(1)=1$ for some $i\geq 1$.
  Hence 
  \[
    1=\phi(\phi^{i-1}(1))=\ol{\tau}^{-1}\left(1+\frac{r}{\phi^{i-1}(1)}\right).
  \]
  Applying $\ol{\tau}$, we obtain
  \[
    1=1+\frac{r}{\phi^{i-1}(1)},
  \]
  and so $r=0$, a contradiction. Therefore we conclude $1\notin H_\infty$.
  Secondly, we assume $1\in H_0$, which means $\phi^n(1)=0$ for some $n\geq 1$.
  Let $F'=F(\ol{\tau}\phi(1),\dots,\ol{\tau}^{n-1}\phi^{n-1}(1))$.
  Since $F'$ is a subfield of $\ol{F}$, the extension $F'/F$ is algebraic.
  Hence $F'/K$ is an algebraic function field of one variable.
  There exists a place $Q$ of $F'/K$ such that $Q\cap F=P$.
  Let $m$ be the maximum of $0\leq i\leq n-1$ such that
  $v_Q(\ol{\tau}^i\phi^i(1))=0$.
  This is possible because $\ol{\tau}^0\phi^0(1)=1$.
  Since $\phi^m(1)\neq 0$, we find
  \[
    \phi^{m+1}(1)=\ol{\tau}^{-1}\left(1+\frac{r}{\phi^m(1)}\right),
  \]
  and so
  \[
    \ol{\tau}^{m+1}\phi^{m+1}(1)
    =\ol{\tau}^m\left(1+\frac{r}{\phi^m(1)}\right)
    =1+\frac{\tau^m r}{\ol{\tau}^m\phi^m(1)}\in F'.
  \]
  Note $v_Q(\tau^mr)>0$ and $v_Q(\ol{\tau}^m\phi^m(1))=0$.
  Then, from the above equation,
  we obtain $v_Q(\ol{\tau}^{m+1}\phi^{m+1}(1))=0$, 
  which contradicts the definition of $m$.
  Therefore we conclude $1\notin H_0$.
  
  The first result and the second imply $1\notin H$.
  Hence we find $\lambda=m_0=0$,
  and from the equation \eqref{2017-04-25p2a}
  \[
    \prod_{i=1}^l(Y-\mu_i)^{m_i}
    =\prod_{i=1}^l\left(Y-\frac{r}{\ol{\tau}\mu_i-1}\right)^{m_i}.
  \]
  This implies all the $\mu_i$ are non-zero, for $r$ is non-zero.
  Hence, from the equation \eqref{2017-04-25p2b}, it follows that
  \[
    H=\{\mu_1,\dots,\mu_l\}\not\ni 0.
  \]
  By definition, $H_0$ must be empty. 
  Therefore we find $H_\infty=H$, 
  and that $\phi=\phi|_{H_\infty}$ is a bijection of $H$ to $H$.
  Finally we assume $H\neq\emptyset$.
  Let $\eta=\mu_1$ for brevity.
  There exists some $j\geq 1$ such that $\phi^j(\eta)=\eta$.
  On the other hand, for any $i\geq 1$, 
  \[
    \phi^i(\eta)=\phi(\phi^{i-1}(\eta))=\ol{\tau}^{-1}\left(1+\frac{r}{\phi^{i-1}(\eta)}\right).
  \]
  Applying $\ol{\tau}^i$, we obtain
  \[
    \ol{\tau}^i\phi^i(\eta)=1+\frac{\tau^{i-1}r}{\ol{\tau}^{i-1}\phi^{i-1}(\eta)},
  \]
  which is equivalent to
  $\eq(\ol{\tau}^i\phi^i(\eta);\tau^{i-1}A;\ol{\tau}^{i-1}\phi^{i-1}(\eta))$.
  Hence by Lemma \ref{2017-04-25pA2} 
  \[
    \eq(\ol{\tau}^j\phi^j(\eta)=\ol{\tau}^j\eta;(\tau^{j-1}A)(\tau^{j-2}A)\cdots A;\eta)
  \]
  holds. It follows that $\eta\in\ol{\mathcal{F}}$ is a solution of
  $\eq(A,j)/\mathcal{F}$, a contradiction. Therefore we conclude $H=\emptyset$,
  which implies $N\in F^\times$.
\end{proof}

\begin{lemma}\label{2017-04-25p6a}
  If $R\in F(Y)$ satisfies 
  \begin{equation}\label{2017-04-25p6b}
    R-\alpha\left(-\frac{r}{Y^2}\right)^\gamma\tau(R)=0,\quad
    \alpha\in F^\times,\ \gamma\in\mathbb{Z}_{>0},
  \end{equation}
  then $R=0$.
\end{lemma}

\begin{proof}
  Assume $R\neq 0$. We shall derive a contradiction.
  Let $R=M/N$, where $M,N\in F[Y]\setminus\{0\}$ are relatively prime,
  and let $\mu=\deg M$ and $\nu=\deg N$.
  From the equation \eqref{2017-04-25p6b}, we obtain
  \[
    \frac{M}{N}=\alpha\left(-\frac{r}{Y^2}\right)^\gamma\frac{\tau(M)}{\tau(N)}
    =\alpha\left(-\frac{r}{Y^2}\right)^\gamma\frac{Y^\nu}{Y^\mu}
    \cdot\frac{Y^\mu\tau(M)}{Y^\nu\tau(N)},
  \]
  and so 
  \[
    MY^{2\gamma}Y^\mu\cdot Y^\nu\tau(N)=N\alpha(-r)^\gamma Y^\nu\cdot Y^\mu\tau(M).
  \]
  Here, $Y\nmid Y^\mu\tau(M)$ and $Y\nmid Y^\nu\tau(N)$ hold.
  Furthermore, $Y^\mu\tau(M)$ and $Y^\nu\tau(N)$ are relatively prime.
  Indeed, there exist $A,B\in F[Y]$ such that $AM+BN=1$.
  Applying $\tau$, we obtain $\tau(A)\tau(M)+\tau(B)\tau(N)=1$, and so
  \[
    Y^{m-\mu}\tau(A)\cdot Y^\mu\tau(M)+Y^{m-\nu}\tau(B)\cdot Y^\nu\tau(N)=Y^m,
  \]
  where $m=\max\{\mu,\nu\}+\max\{\deg A,\deg B\}$.
  Hence $\gcd(Y^\mu\tau(M),Y^\nu\tau(N))$ divides into $Y^m$,
  which implies $\gcd(Y^\mu\tau(M),Y^\nu\tau(N))=1$.
  Therefore we find 
  \[
    M\mid Y^\nu\cdot Y^\mu\tau(M), \quad Y^\mu\tau(M)\mid M,
  \]
  and
  \[
    N\mid Y^{2\gamma}Y^\mu\cdot Y^\nu\tau(N),\quad Y^\nu\tau(N)\mid N.
  \]
  
  Let $Y^{\lambda}\| M$. Since $M/Y^{\lambda}$ and $Y^\mu\tau(M)$
  divide into each other, there exists $\beta_1\in F^\times$ such that
  $M/Y^{\lambda}=\beta_1Y^\mu\tau(M)$ or
  \[
    M=\beta_1Y^{\lambda}\cdot Y^\mu\tau(M).
  \]
  By Lemma \ref{2017-04-25p1a}, we find $M\in F^\times$.
  Similarly, we also find $N\in F^\times$.
  Therefore we obtain $R\in F^\times$, which contradicts the equation \eqref{2017-04-25p6b}.
\end{proof}

\begin{lemma}\label{2017-04-25p9a}
  There is no $S\in F(Y)$ such that
  \begin{equation}\label{2017-04-25p9b}
    S+\frac{r}{Y^2}\tau(S)+\frac{\alpha}{Y}=0,\quad \alpha\in F^\times.
  \end{equation}
\end{lemma}

\begin{proof}
  Assume there exists such $S\in F(Y)$, which must be non-zero and can be
  written as
  \[
    S=\frac{M}{N},
  \]
  where $M,N\in F[Y]\setminus\{0\}$ are relatively prime.
  Let $\mu=\deg M$ and $\nu=\deg N$.
  We have
  \[
    \frac{M}{N}=-\frac{\alpha}{Y}-\frac{r}{Y^2}\cdot\frac{\tau(M)}{\tau(N)}
    =-\frac{\alpha}{Y}-\frac{r}{Y^2}\cdot\frac{Y^\nu}{Y^\mu}
    \cdot\frac{Y^\mu\tau(M)}{Y^\nu\tau(N)},
  \]
  from which we obtain
  \[
    MY^2Y^\mu\cdot Y^\nu\tau(N)
    =-N(\alpha YY^\mu\cdot Y^\nu\tau(N)+rY^\nu\cdot Y^\mu\tau(M)).
  \]
  The polynomials $Y^\mu\tau(M)$ and $Y^\nu\tau(N)$ do not have a factor $Y$.
  Moreover, they are relatively prime.
  Hence we find $N\mid Y^2Y^\mu\cdot Y^\nu\tau(N)$ and $Y^\nu\tau(N)\mid N$.
  Let $Y^\lambda\|N$.
  Since $N/Y^\lambda$ and $Y^\nu\tau(N)$ divide into each other,
  there exists $\beta\in F^\times$ such that $N/Y^\lambda=\beta Y^\nu\tau(N)$.
  By Lemma \ref{2017-04-25p1a}, we find $N\in F^\times$, and so $S\in F[Y]$.
  From the equation \eqref{2017-04-25p9b}, it follows that
  \[
    S=-\frac{\alpha}{Y}-\frac{r}{Y^2}\tau(S)\in F[1/Y].
  \]
  Since $S\in F[Y]\cap F[1/Y]$ implies $S\in F$, we find a contradiction
  from the above equation.
\end{proof}

\begin{lemma}\label{2017-04-25p11a}
  If there exists $S\in F(Y)$ such that
  \begin{equation}\label{2017-04-25p11b}
    S+\alpha Y^2\tau(S)+\beta Y=0,\quad \alpha,\beta\in F^\times,
  \end{equation}
  then there exists $a\in F$ such that
%  \[
%  \begin{aligned}
%    &\tau^2(r^2\alpha)\tau^3(a)+(\tau(r)+1)\tau^2(a)\\
%    &-\left(\frac{1}{\tau(\alpha)}+\frac{1}{\tau(r\alpha)}\right)\tau(a)
%    -\frac{1}{\tau(r\alpha)\alpha}a+\frac{\tau(\beta)}{\tau(r\alpha)}=0.
%  \end{aligned}
%  \]
  \[
  \begin{aligned}
    &\tau^2(r^2\alpha)\tau(r\alpha)\alpha\tau^3(a)
    +(\tau(r)+1)\tau(r\alpha)\alpha\tau^2(a)\\
    &\quad -(\tau(r)+1)\alpha\tau(a)-a+\alpha\tau(\beta)=0.
  \end{aligned}
  \]
\end{lemma}

\begin{proof}
  Since $S$ is non-zero, it can be written as
  \[
    S=\frac{M}{N},
  \]
  where $M,N\in F[Y]\setminus\{0\}$ are relatively prime.
  Let $\mu=\deg M$ and $\nu=\deg N$.
  We have
  \[
    \frac{M}{N}=-\beta Y-\alpha Y^2\cdot\frac{\tau(M)}{\tau(N)}
    =-\beta Y-\alpha Y^2\cdot\frac{Y^\nu}{Y^\mu}\cdot\frac{Y^\mu\tau(M)}{Y^\nu\tau(N)},
  \]
  from which we obtain
  \[
    MY^\mu\cdot Y^\nu\tau(N)
    =-N(\beta YY^\mu\cdot Y^\nu\tau(N)+\alpha Y^2Y^\nu\cdot Y^\mu\tau(M)).
  \]
  The polynomials $Y^\mu\tau(M)$ and $Y^\nu\tau(N)$ do not have a factor $Y$.
  Moreover, they are relatively prime. Hence we find
  $N\mid Y^\mu\cdot Y^\nu\tau(N)$ and $Y^\nu\tau(N)\mid N$.
  Let $Y^\lambda\|N$. Since $N/Y^\lambda$ and $Y^\nu\tau(N)$ divide into
  each other, there exists $\gamma\in F^\times$ such that
  $N/Y^\lambda=\gamma Y^\nu\tau(N)$. By Lemma \ref{2017-04-25p1a}, we find
  $N\in F^\times$, and so $S\in F[Y]$.
  Looking at degrees in the equation \eqref{2017-04-25p11b},
  we see $\deg S=1$ or $2$.
  Hence $S$ can be written as
  \[
    S=aY^2+bY+c,\quad a,b,c\in F.
  \]
  From the equation \eqref{2017-04-25p11b}, we obtain
  \[
  \left\{
  \begin{aligned}
    &a+\alpha(\tau(a)+\tau(b)+\tau(c))=0,\\
    &b+\alpha(2r\tau(a)+r\tau(b))+\beta=0,\\
    &c+\alpha r^2\tau(a)=0.
  \end{aligned}
  \right.
  \]
  From the first equation and the second, it follows that
  \[
    \tau(b)=-\tau(c)-\tau(a)-\frac{a}{\alpha}
    =\tau(r^2\alpha)\tau^2(a)-\tau(a)-\frac{a}{\alpha}.
  \]
  Hence, from the second equation, we find
  \[
    \tau(b)+2\tau(r\alpha)\tau^2(a)+\tau(r\alpha)\tau^2(b)+\tau(\beta)=0,
  \]
  and so
  \[
  \begin{aligned}
    &\tau(r^2\alpha)\tau^2(a)-\tau(a)-\frac{a}{\alpha}+2\tau(r\alpha)\tau^2(a)\\
    &\quad +\tau(r\alpha)\left(\tau^2(r^2\alpha)\tau^3(a)-\tau^2(a)-\frac{\tau(a)}{\tau(\alpha)}\right)
    +\tau(\beta)=0.
  \end{aligned}
  \]
  Multiplying it by $\alpha$, we obtain
  \[
  \begin{aligned}
    &\tau^2(r^2\alpha)\tau(r\alpha)\alpha\tau^3(a)
    +(\tau(r)+1)\tau(r\alpha)\alpha\tau^2(a)\\
    &\quad -(\tau(r)+1)\alpha\tau(a)-a+\alpha\tau(\beta)=0,
  \end{aligned}
  \]
  the required.
\end{proof}

\section{Proof of Theorem}\label{theorem}

\begin{theorem}\label{2017-04-28p1a}
  Let $\mathcal{F}=(F,D_0,\tau_0)$ be a DTC field with
  $D_0\tau_0=s\tau_0D_0$ for a certain $s\in F^\times$,
  and $F/K$ an algebraic function filed of one variable.
  Let
  \[
    A=\begin{pmatrix}1 & r\\ 1&0\end{pmatrix}\in \mathrm{M}_2(F),\quad
    r\neq 0,\ D_0r\neq 0,
  \]
  and suppose there exists a place $P$ of $F/K$ such that
  $v_P(\tau_0^ir)>0$ for all $i\geq 0$.
  Additionally, suppose that for any $i\geq 1$, $\eq(A,i)/(F,\tau_0)$ has
  no solution algebraic over $F$.
%  If for a certain DTC overfield $\mathcal{U}=(U,D,\tau)$ of $\mathcal{F}$ with 
%  $D\tau=s\tau D$, 
  Let $\mathcal{U}=(U,D,\tau)$ be a DTC overfield of $\mathcal{F}$ with
  $D\tau=s\tau D$.
  If there exists a solution 
  $f\in\mathcal{U}$ of $\eq(A,1)/(F,\tau_0)$ which is differentially algebraic
  over $F$, then there exists $g\in F$ such that
  \[
  \begin{aligned}
    &\tau^2(sr)\tau(s)s\tau^3(g)+(\tau(r)+1)\tau(s)s\tau^2(g)\\
    &\quad -(\tau(r)+1)s\tau(g)-rg+s\tau(D(r)/r)=0.
  \end{aligned}
  \]
\end{theorem}

\begin{proof}
%  Let $n=\trdeg F(f,f',f'',\dots)/F$, where
%  we use the ordinary notation $f'$ instead of $Df$ for brevity.
%  Since $f$ is transcendental over $F$,
%  $f',\dots,f^{(n)}$ are algebraically dependent over $F(f)$.
%  Hence there exists an irreducible polynomial
%  $G\in F(f)[Y_1,\dots,Y_n]\setminus F(f)$
%  such that $G(f',f'',\dots,f^{(n)})=0$.
  We use the ordinary notation $f'$ instead of $Df$ for brevity.
  From $\tau(f)f=f+r$ and $f\neq 0$, we have
  \[
    \tau f=1+\frac{r}{f}.
  \]
  By differentiation, we obtain
  \begin{equation}\label{2017-04-28p2a}
    s\tau f'=D\tau f=\frac{r'}{f}-\frac{r}{f^2}f'.
  \end{equation}
  By differentiation again, we obtain
  \[
    s'\tau f'+sD\tau f'=\frac{r''f-r'f'}{f^2}-\frac{(r'f'+rf'')f^2-rf'\cdot 2ff'}{f^4},
  \]
  and so
  \begin{equation}\label{2017-04-28p2b}
    s^2\tau f''=-s'\tau f'+\frac{r''}{f}-\frac{2r'}{f^2}f'+\frac{2r}{f^3}f'^2
    -\frac{r}{f^2}f''.
  \end{equation}
  By repeated differentiation, it is seen that for all $m\geq 3$,
  \begin{equation}\label{2017-04-28p3a}
  \begin{aligned}
    s^m\tau f^{(m)}&=(\text{a polynomial in $f',\dots,f^{(m-2)}$ over $F(f)$})\\
    &\quad +(\text{an element of $F(f)$})f^{(m-1)}
    +\frac{2mr}{f^3}f'f^{(m-1)}-\frac{r}{f^2}f^{(m)}.
  \end{aligned}
  \end{equation}
  Recall $\trdeg F(f,f',f'',\dots)/F<\infty$ and that $f$ is transcendental
  over $F$.
  Let $n\geq 1$ be the minimum number such that
  $f',f'',\dots,f^{(n)}$ are algebraically dependent over $F(f)$.
  If $n\geq 2$, $f',f'',\dots, f^{(n-1)}$ are algebraically independent
  over $F(f)$.
  Define the polynomials $Z_1,\dots,Z_n\in F(f)[Y_1,\dots,Y_{n-1}]$ by
  \[
    s^m\tau f^{(m)}+\frac{r}{f^2}f^{(m)}=Z_m(f',\dots,f^{(n-1)}).
  \]
  We have $Z_m\in F(f)[Y_1,\dots,Y_{m-1}]$.
  Let $G\in F(f)[Y_1,\dots,Y_n]\setminus F(f)$ be an irreducible polynomial
  such that $G(f',\dots,f^{(n)})=0$, and $H$ a polynomial defined by
  \[
    H=G^\tau(s^{-1}(Z_1-(r/f^2)Y_1),\dots,s^{-n}(Z_n-(r/f^2)Y_n))
    \in F(f)[Y_1,\dots,Y_n],
  \]
  where $G^\tau$ is the polynomial whose coefficients are the first transforms
  of the corresponding coefficients of $G$.
  It follows that
  \[
  \begin{aligned}
    H(f',\dots,f^{(n)})
    &=G^\tau(s^{-1}(Z_1(f',\dots,f^{(n-1)})-(r/f^2)f'),\dots,\\
    &\hspace{4em} s^{-n}(Z_n(f',\dots,f^{(n-1)})-(r/f^2)f^{(n)}))\\
    &=G^\tau(\tau f',\dots,\tau f^{(n)})\\
    &=\tau(G(f',\dots,f^{(n)}))\\
    &=0.
  \end{aligned}
  \]
  Hence we find $G\mid H$ (cf. the book \cite{ZS1958}, Ch. II, \S 13, Lemma 2).

  The polynomial $G$ is expressed as
  \[
    G=\sum_{i=(i_1,\dots,i_n)}R_iY_1^{i_1}\cdots Y_n^{i_n},
    \quad R_i\in F(f).
  \]
  Let $d$ be the maximum of $i$ with $R_i\neq 0$ in the sense of
  \[
    (i_1,\dots,i_n)<(j_1,\dots,j_n)
    \Longleftrightarrow i_n=j_n,\dots,i_{m+1}=j_{m+1},\ i_m<j_m.
  \]
  We may suppose $R_d=1$.
  If we let $\|i\|$ denote $i_1+i_2+\dots+i_n$ and $[i]$ denote $i_1+2i_2+\dots+ni_n$,
  $H$ can be written as 
  \[
  \begin{aligned}
    H&=\sum_i\tau(R_i)\left(s^{-1}\left(Z_1-\frac{r}{f^2}Y_1\right)\right)^{i_1}
    \cdots\left(s^{-n}\left(Z_n-\frac{r}{f^2}Y_n\right)\right)^{i_n}\\
    &=\sum_i\tau(R_i)s^{-[i]}\left(\left(-\frac{r}{f^2}\right)^{\|i\|}Y_1^{i_1}\cdots Y_n^{i_n}+\cdots\right)\\
    &=s^{-[d]}\left(-\frac{r}{f^2}\right)^{\|d\|}Y_1^{d_1}\cdots Y_n^{d_n}
    +\cdots.
  \end{aligned}
  \]
  Hence we find
  \[
    H=s^{-[d]}\left(-\frac{r}{f^2}\right)^{\|d\|}G.
  \]
  Comparing the coefficients of $Y_1^{i_1}\cdots Y_n^{i_n}$, we obtain
  \begin{equation}\label{2017-04-28p5a}
    s^{-[d]}\left(-\frac{r}{f^2}\right)^{\|d\|}R_i
    =\tau(R_i)s^{-[i]}\left(-\frac{r}{f^2}\right)^{\|i\|}
    +\sum_{j>i}\tau(R_j)s^{-[j]}S_{ji},
  \end{equation}
  where $S_{ji}\in F(f)$ is the coefficient of $Y_1^{i_1}\cdots Y_n^{i_n}$
  of
  \[
    \left(Z_1-\frac{r}{f^2}Y_1\right)^{j_1}\cdots
    \left(Z_n-\frac{r}{f^2}Y_n\right)^{j_n}.
  \]
  Let $m=\min\{k\;|\;d_k\neq 0\}$.
  We shall show $m=1$.
  
  Firstly, assume $m\geq 3$.
  We have $d=(0,\dots,0,d_m,\dots,d_n)$ and
  \[
  \begin{aligned}
    &\left(Z_m-\frac{r}{f^2}Y_m\right)^{d_m}\cdots
    \left(Z_n-\frac{r}{f^2}Y_n\right)^{d_n}\\
    &=\left(\dots+\frac{2mr}{f^3}Y_1Y_{m-1}-\frac{r}{f^2}Y_m\right)^{d_m}\cdots
    \left(\dots+\frac{2nr}{f^3}Y_1Y_{n-1}-\frac{r}{f^2}Y_n\right)^{d_n}\\
    &=\left(-\frac{r}{f^2}\right)^{\|d\|}Y_m^{d_m}\cdots Y_n^{d_n}\\
    &\quad +d_m\frac{2mr}{f^3}Y_1Y_{m-1}\left(-\frac{r}{f^2}\right)^{\|d\|-1}Y_m^{d_m-1}Y_{m+1}^{d_{m+1}}\cdots Y_n^{d_n}
    +\cdots.
  \end{aligned}
  \]
  Let $e=(1,0,\dots,0,1,d_m-1,d_{m+1},\dots,d_n)$.
  For any $e<i<d$, $R_i=0$ holds. 
  Indeed, let $i$ be the maximum of $e<i<d$ with $R_i\neq 0$ if they exist.
  From the equation \eqref{2017-04-28p5a}, we have
  \[
    s^{-[d]}\left(-\frac{r}{f^2}\right)^{\|d\|}R_i
    =\tau(R_i)s^{-[i]}\left(-\frac{r}{f^2}\right)^{\|i\|}+s^{-[d]}S_{di}.
  \]
  Since we find $S_{di}=0$ from the above equation, this implies
  \[
    R_i-s^{[d]-[i]}\left(-\frac{r}{f^2}\right)^{\|i\|-\|d\|}\tau(R_i)=0.
  \]
  Here, the form of $i$ is as follows:
  \[
    i=(\ast,\dots,\ast,\ast\geq 1,d_m-1,d_{m+1},\dots,d_n),\quad 
    \|(\ast,\dots,\ast,\ast\geq 1)\|\geq 2.
  \]
  Hence we obtain $\|i\|-\|d\|\geq 1$, and so
  $R_i=0$ by Lemma \ref{2017-04-25p6a}, a contradiction.
  We found $R_i=0$ for any $e<i<d$.
  From the equation \eqref{2017-04-28p5a}, we have
  \[
    s^{-[d]}\left(-\frac{r}{f^2}\right)^{\|d\|}R_e
    =\tau(R_e)s^{-[e]}\left(-\frac{r}{f^2}\right)^{\|e\|}+s^{-[d]}S_{de},
  \]
  where
  \[
    S_{de}=d_m\frac{2mr}{f^3}\left(-\frac{r}{f^2}\right)^{\|d\|-1}.
  \]
  By $\|e\|-\|d\|=1$ and
  \[
    [d]-[e]=md_m-(1+(m-1)+m(d_m-1))=0,
  \]
  it follows that
  \[
    R_e+\frac{r}{f^2}\tau(R_e)+\frac{2md_m}{f}=0,
  \]
  which contradicts Lemma \ref{2017-04-25p9a}.
  Therefore we conclude $m\leq 2$.
  
  Secondly, assume $m=2$.
  We have $d=(0,d_2,\dots,d_n)$ and
  \[
  \begin{aligned}
    &\left(Z_2-\frac{r}{f^2}Y_2\right)^{d_2}\cdots
    \left(Z_n-\frac{r}{f^2}Y_n\right)^{d_n}\\
    &=\left(\dots+\frac{2r}{f^3}Y_1^2-\frac{r}{f^2}Y_2\right)^{d_2}
    \left(\dots+\frac{6r}{f^3}Y_1Y_{2}-\frac{r}{f^2}Y_3\right)^{d_3}\\
    &\quad\cdots
    \left(\dots+\frac{2nr}{f^3}Y_1Y_{n-1}-\frac{r}{f^2}Y_n\right)^{d_n}\\
    &=\left(-\frac{r}{f^2}\right)^{\|d\|}Y_2^{d_2}\cdots Y_n^{d_n}\\
    &\quad +d_2\frac{2r}{f^3}Y_1^2\left(-\frac{r}{f^2}\right)^{\|d\|-1}Y_2^{d_2-1}Y_{3}^{d_{3}}\cdots Y_n^{d_n}
    +\cdots.
  \end{aligned}
  \]
  Let $e=(2,d_2-1,d_3,\dots,d_n)$.
  For any $e<i<d$, $R_i=0$ holds. 
  Indeed, let $i$ be the maximum of $e<i<d$ with $R_i\neq 0$ if they exist.
  From the equation \eqref{2017-04-28p5a}, we have
  \[
  \begin{aligned}
    s^{-[d]}\left(-\frac{r}{f^2}\right)^{\|d\|}R_i
    &=\tau(R_i)s^{-[i]}\left(-\frac{r}{f^2}\right)^{\|i\|}+s^{-[d]}S_{di}\\
    &=\tau(R_i)s^{-[i]}\left(-\frac{r}{f^2}\right)^{\|i\|},
  \end{aligned}
  \]
  and so
  \[
    R_i-s^{[d]-[i]}\left(-\frac{r}{f^2}\right)^{\|i\|-\|d\|}\tau(R_i)=0.
  \]
  Here, the form of $i$ is as follows:
  \[
    i=(\ast\geq 3,d_2-1,d_{3},\dots,d_n).
  \]
  Hence we obtain $\|i\|-\|d\|\geq 2$, and so
  $R_i=0$ by Lemma \ref{2017-04-25p6a}, a contradiction.
  We found $R_i=0$ for any $e<i<d$.
  From the equation \eqref{2017-04-28p5a}, we have
  \[
    s^{-[d]}\left(-\frac{r}{f^2}\right)^{\|d\|}R_e
    =\tau(R_e)s^{-[e]}\left(-\frac{r}{f^2}\right)^{\|e\|}
    +s^{-[d]}d_2\frac{2r}{f^3}\left(-\frac{r}{f^2}\right)^{\|d\|-1}.
  \]
  By $\|e\|-\|d\|=1$ and
  $[d]-[e]=2d_2-(2+2(d_2-1))=0$,
  it follows that
  \[
    R_e+\frac{r}{f^2}\tau(R_e)+\frac{2d_2}{f}=0,
  \]
  which contradicts Lemma \ref{2017-04-25p9a}.
  Therefore we conclude $m=1$.
  
  We have $d=(d_1,d_2,\dots,d_n)$, $d_1\neq 0$, and
  \[
  \begin{aligned}
    &\left(Z_1-\frac{r}{f^2}Y_1\right)^{d_1}\cdots
    \left(Z_n-\frac{r}{f^2}Y_n\right)^{d_n}\\
    &=\left(\frac{r'}{f}-\frac{r}{f^2}Y_1\right)^{d_1}
    \left(\dots+\frac{2r}{f^3}Y_1^2-\frac{r}{f^2}Y_2\right)^{d_2}
    \left(\dots+\frac{6r}{f^3}Y_1Y_{2}-\frac{r}{f^2}Y_3\right)^{d_3}\\
    &\quad\cdots
    \left(\dots+\frac{2nr}{f^3}Y_1Y_{n-1}-\frac{r}{f^2}Y_n\right)^{d_n}\\
    &=\left(-\frac{r}{f^2}\right)^{\|d\|}Y_1^{d_1}\cdots Y_n^{d_n}
    +d_1\frac{r'}{f}\left(-\frac{r}{f^2}\right)^{\|d\|-1}Y_1^{d_1-1}Y_{2}^{d_{2}}\cdots Y_n^{d_n}
    +\cdots.
  \end{aligned}
  \]
  Let $e=(d_1-1,d_2,\dots,d_n)$.
  From the equation \eqref{2017-04-28p5a}, we have
  \[
    s^{-[d]}\left(-\frac{r}{f^2}\right)^{\|d\|}R_e
    =\tau(R_e)s^{-[e]}\left(-\frac{r}{f^2}\right)^{\|e\|}
    +s^{-[d]}d_1\frac{r'}{f}\left(-\frac{r}{f^2}\right)^{\|d\|-1}.
  \]
  By $\|e\|-\|d\|=-1$ and
  $[d]-[e]=d_1-(d_1-1)=1$,
  it follows that
  \[
    R_e+\frac{s}{r}f^2\tau(R_e)+\frac{d_1r'}{r}f=0,\quad r'\neq 0.
  \]
  By Lemma \ref{2017-04-25p11a}, there exists $a\in F$ such that
  \[
  \begin{aligned}
    &\tau^2(sr)\tau(s)\frac{s}{r}\tau^3(a)+(\tau(r)+1)\tau(s)\frac{s}{r}\tau^2(a)\\
    &\quad -(\tau(r)+1)\frac{s}{r}\tau(a)-a+\frac{s}{r}\tau\left(\frac{d_1r'}{r}\right)
    =0.
  \end{aligned}
  \]
  Letting $g=a/d_1\in F$ and multiplying by $r/d_1$, we obtain
  \[
  \begin{aligned}
    &\tau^2(sr)\tau(s)s\tau^3(g)+(\tau(r)+1)\tau(s)s\tau^2(g)\\
    &\quad -(\tau(r)+1)s\tau(g)-rg+s\tau(r'/r)=0,
  \end{aligned}
  \]
  the required.
\end{proof}

\begin{proposition}\label{2017-05-25p1a}
  Let $\mathcal{F}=(F,D_0,\tau_0)$ be a DTC field with 
  $D_0\tau_0=s\tau_0D_0$ for a certain $s\in F^\times$.
  Let 
  \[
    A=\begin{pmatrix}1&r\\1&0\end{pmatrix}\in \mathrm{M}_2(F),\quad r\neq 0.
  \]
  Suppose there exists $g\in F$ such that
  \[
  \begin{aligned}
    &\tau_0^2(sr)\tau_0(s)s\tau_0^3(g)+(\tau_0(r)+1)\tau_0(s)s\tau_0^2(g)\\
    &\quad -(\tau_0(r)+1)s\tau_0(g)-rg+s\tau_0(D_0(r)/r)=0,
  \end{aligned}
  \]
  and let
  \[
    R=gY^2-(\tau_0(sr)s\tau_0^2(g)+s\tau_0(g)-rg+D_0(r)/r)Y-sr\tau_0(g)\in F[Y].
  \]
  Then there exist a DTC overfield $\mathcal{U}=(U,D,\tau)$ of $\mathcal{F}$
  with $D\tau=s\tau D$ and a solution $f\in \mathcal{U}$ of $\eq(A,1)/(F,\tau_0)$
  which satisfies the differential Riccati equation,
  \[
    Df+R(f)=0.
  \]
\end{proposition}

\begin{proof}
  Let $f$ be a transcendental element over $F$ and
  $D$ the derivation of $F(f)$ which is an extension of $D_0$ with $Df=-R(f)$.
  Moreover, let $\tau$ be an isomorphism of $F(f)$ into $F(f)$ such that
  $\tau|_F=\tau_0$ and $\tau f=1+r/f$.
  We shall prove that $\mathcal{U}=(F(f),D,\tau)$ is a DTC overfield of
  $\mathcal{F}$ with $D\tau=s\tau D$.
  Let $R=gY^2+bY-sr\tau(g)$ for brevity.
  We have
  \[
    -s\tau b=\tau^2(sr)\tau(s)s\tau^3(g)+\tau(s)s\tau^2(g)
    -\tau(r)s\tau(g)+s\tau(D(r)/r),
  \]
  and so by the definition of $g$,
  \[
    s\tau b=\tau(r)\tau(s)s\tau^2(g)-s\tau(g)-rg
  \]
  or
  \[
    b+s\tau b=-2s\tau(g)-\frac{Dr}{r}.
  \]
  Hence, by a straightforward calculation, it follows that
  \[
  \begin{aligned}
    &R(f)+\frac{s}{r}f^2\tau(R(f))+\frac{Dr}{r}f\\
    &=gf^2+bf-sr\tau(g)\\
    &\quad+\frac{s}{r}f^2\left\{\tau(g)\left(1+\frac{r}{f}\right)^2+\tau(b)\left(1+\frac{r}{f}\right)-\tau(sr)\tau^2(g)\right\}
    +\frac{Dr}{r}f\\
    &=\left\{g+\frac{s}{r}(\tau(g)+\tau(b)-\tau(sr)\tau^2(g))\right\}f^2\\
    &\quad +\left\{b+\frac{s}{r}(2r\tau(g)+r\tau(b))+\frac{Dr}{r}\right\}f\\
    &=0.
  \end{aligned}
  \]
  Substituting $R(f)=-Df$, we obtain
  \[
    -D(f)-\frac{s}{r}f^2\tau(D(f))+\frac{Dr}{r}f=0,
  \]
  and so
  \[
    s\tau Df=-\frac{r}{f^2}D(f)+\frac{Dr}{f}
    =D\left(1+\frac{r}{f}\right)=D\tau f,
  \]
  which implies $D\tau=s\tau D$.
\end{proof}

\section{Application}\label{application}

In this section, we shall investigate differential transcendence of
solutions of the Riccati equation,
\[
  y(qt)=-qt+\frac{1}{y},
\]
which is associated with the $q$-Airy equation,
\[
  y(q^2t)+qty(qt)-y(t)=0.
\]

Let $C$ be an algebraically closed field and $C(t)$ a rational function field.
Let $\mathcal{L}=(C(t),D_0,\tau_0)$ be a DTC field with
\[D_0|_C=0,\ D_0t=1,\quad \tau_0|_C=\mathrm{id},\ \tau_0t=qt,\ q\in C^\times.\]
We use 
\[
  A=\begin{pmatrix}-qt & 1\\ 1&0\end{pmatrix}\in \mathrm{GL}_2(C(t))
\]
as a matrix associated with the above Riccati equation.
By the result introduced in the author's paper \cite{Sei2017}, we have
the following.

\begin{lemma}\label{2017-05-20p1a}
  If $q$ is not a root of unity,
  then for all $i\geq 1$, $\eq(A,i)/(C(t),\tau_0)$ has no solution
  algebraic over $C(t)$.
\end{lemma}

In his paper \cite{Sei2016}, he defined difference field extensions of
valuation ring type, and introduced several results as follows.

\begin{definition}\label{2014-08-14p1a}
  Let $\mathcal{N}/\mathcal{K}$ be a difference field extension,
  where $\mathcal{N}=(N,\tau)$.
  We say that $\mathcal{N}/\mathcal{K}$ is of valuation ring type
  if there exists a chain of difference field extension,
  \[
    \mathcal{K}=\mathcal{K}_0\subset\mathcal{K}_1\subset\dots\subset
    \mathcal{K}_{n-1}\subset\mathcal{K}_n=\mathcal{N},
  \]
  such that each $\mathcal{K}_i/\mathcal{K}_{i-1}$
  satisfies one of the following.
  \begin{enumerate}
  \item $K_i/K_{i-1}$ is algebraic.
  \item $K_i/K_{i-1}$ is an algebraic function field of one variable,
  and there exists a place $P$ of $K_i/K_{i-1}$ such that
  $\tau^jP\subset P$ for some $j\in\mathbb{Z}_{>0}$.
  \end{enumerate}
\end{definition}

\begin{lemma}[Corollary 6 in \cite{Sei2016}]\label{2014-08-14p7a}
  Let $\mathcal{K}$ be a difference field, and 
  $f$ a solution of $y_1=ay+b$, $a,b\in K$, $a\neq 0$,
  transcendental over $K$.
  Then $\mathcal{K}\gen{f}/\mathcal{K}$ is of valuation ring type.
\end{lemma}

\begin{remark}
  A chain of difference field extensions of valuation ring type
  is also of valuation ring type.
\end{remark}

\begin{lemma}[Theorem 8 in \cite{Sei2016}]\label{2017-05-20p1b}
  Let $\mathcal{M}=(M,\tau)$ be a difference overfield of $(C(t),\tau_0)$.
  Let $k\in\mathbb{Z}_{>0}$, and suppose that
  $\eq(A,k)/\mathcal{M}$ has a solution in a certain difference field extension
  $\mathcal{N}/\mathcal{M}$ of valuation ring type.
  Then $\eq(A,ki)/\mathcal{M}$ has a solution in $\ol{\mathcal{M}}$
  for some $i\in\mathbb{Z}_{>0}$, where $\ol{\mathcal{M}}$ is 
  the algebraic closure of $\mathcal{M}$ in an algebraic closure 
  $\ol{\mathcal{N}}$ of $\mathcal{N}$.
\end{lemma}

Here, we shall prove the following.

\begin{theorem}\label{2017-05-20p2a}
  Suppose that $q$ is not a root of unity.
  Let $\mathcal{U}=(U,D,\tau)$ be a DTC overfield of $\mathcal{L}$ with
  $D\tau=q\tau D$, and
  $\mathcal{F}$ a DTC intermediate field of $\mathcal{U}/\mathcal{L}$ such that
  \[
    \{x\in F\;|\; Dx=0\}=\{x\in F\;|\;\tau x=x\}=C
  \]
  and $\mathcal{F}=\mathcal{L}\gen{f_1,\dots,f_n}_\mathcal{U}$, where 
  $f_1,\dots,f_n$ satisfy
  \[
    \tau f_i=\alpha_if_i+\beta_i,\quad 
    \alpha_i,\beta_i\in\mathcal{L}\gen{f_1,\dots,f_{i-1}}_\mathcal{U},
    \ \alpha_i\neq 0,
  \]
  and $S=\{D^jf_i\;|\;1\leq i\leq n,\,j\geq 0\}$ are algebraically independent
  over $C(t)$.
  Then $\eq(A,1)/(C(t),\tau)$ has no solution differentially algebraic 
  over $F$ in $\mathcal{U}$.
\end{theorem}

\begin{proof}
  Let $L_k$ be the underlying field of $\mathcal{L}\gen{f_1,\dots,f_k}_\mathcal{U}$,
  and $S_k$ denote the set $\{D^jf_i\;|\;1\leq i\leq k,\,j\geq 0\}$.
  We shall prove $L_k=C(t)(S_k)$ by induction.
  It is obvious in the case $k=0$. Suppose $k\geq 1$ and that the claim 
  is true for $k-1$.
  By $D^i\tau=q^i\tau D^i$ $(i\geq 1)$ and
  \[
    \tau f_k=\alpha_kf_k+\beta_k,
    \quad \alpha_k,\beta_k\in L_{k-1}=C(t)(S_{k-1}),
  \]
  we find
  \[
  \begin{aligned}
    q^i\tau f_k^{(i)}
    &=D^i(\alpha_kf_k+\beta_k)\\
    &=\alpha_kf_k^{(i)}
    +\text{(an element of $L_{k-1}(f_k,f_k',\dots,f_k^{(i-1)})$)},
%    \in C(t)(S_k),
  \end{aligned}
  \]
  which implies $L_k=L_{k-1}(f_k,f_k',\dots)=C(t)(S_k)$.

  We will find that for all $i\geq 1$, $\eq(A,i)/(F,\tau)$ 
  has no solution algebraic over $F$.
  Let $m$ be the maximum such that
  for all $i\geq 1$, $\eq(A,i)/(L_m,\tau)$
  has no solution algebraic over $L_m$.
  It is well-defined, for the statement is true for $m=0$.
  If $m<n$, then for a certain $k\geq 1$, 
  $\eq(A,k)/(L_{m+1},\tau)$ would have a solution $f$ algebraic
  over $L_{m+1}$.
%  By $D^i\tau=q^i\tau D^i$ $(i\geq 1)$ and
%  \[
%    \tau f_{m+1}=\alpha_{m+1}f_{m+1}+\beta_{m+1},
%    \quad \alpha_{m+1},\beta_{m+1}\in L_m,
%  \]
%  we find
%  \[
%  \begin{aligned}
%    q^i\tau f_{m+1}^{(i)}
%    &=D^i(\alpha_{m+1}f_{m+1}+\beta_{m+1})\\
%    &=\alpha_{m+1}f_{m+1}^{(i)}
%    +\text{(an element of $L_{m}(f_{m+1},f_{m+1}',\dots,f_{m+1}^{(i-1)})$)},
%  \end{aligned}
%  \]
%  which implies $L_{m+1}=L_m(f_{m+1},f_{m+1}',\dots)$.
  Since $f$ is algebraic over $L_{m+1}=L_m(f_{m+1},f_{m+1}',\dots)$,
  $f$ is algebraic over $L_m(f_{m+1},f_{m+1}',\dots,f_{m+1}^{(l)})$
  for some $l$.
  Let $\mathcal{N}$ be the difference field defined by
  \[
    \mathcal{N}=(L_m,\tau)\gen{f_{m+1},\dots,f_{m+1}^{(l)},f}
    \subset(L_{m+1},\tau)\gen{f}.
  \]
  From the above discussion about $\tau f_k^{(i)}$, by Lemma \ref{2014-08-14p7a},
  we see that $\mathcal{N}/(L_m,\tau)$ is of valuation ring type.
  By Lemma \ref{2017-05-20p1b}, 
  there exists $i\in\mathbb{Z}_{>0}$ such that $\eq(A,ki)/(L_m,\tau)$
  has a solution algebraic over $L_m$, which contradicts the definition of $m$.
  Therefore we conclude $m=n$, which proves the above claim.
  
  Let $K=C(S)$.
  We find $F=C(t)(S)=K(t)$ and that $t$ is transcendental over $K$.
  By Proposition \ref{2017-05-11p3a} (ii), we have
  \[
    (\tau P)AP^{-1}=-q^2t\begin{pmatrix}1&(q^3t^2)^{-1}\\1&0\end{pmatrix},
    \quad P=\begin{pmatrix}-qt&1\\-qt&0\end{pmatrix}.
  \]
  Let
  \[
    B=\begin{pmatrix}1&r\\1&0\end{pmatrix},\quad r=\frac{1}{q^3t^2}.
  \]
  For all $i\geq 1$, $\eq(B,i)/(F,\tau)$ has no solution algebraic over $F$.
  Indeed, assume that $\eq(B,k)/(F,\tau)$ has a solution $f$
  algebraic over $F$.
  Let $\mathcal{N}=(F,\tau)\gen{f}$ and
  $A'=(\tau P)AP^{-1}=-q^2tB$.
  Since it follows that
  \[
    A'_k=(\tau^{k-1}A')(\tau^{k-2}A')\cdots A'
    =eB_k,\quad e\in C(t)^\times,
  \]
  $f\in \mathcal{N}$ is a solution of $\eq(A',k)/(F,\tau)$.
  Hence, by Lemma \ref{2017-05-11p1a}, we find that 
  there exists $g\in F(f)$ such that $\eq(g;P^{-1};f)$ holds and 
  $g\in\mathcal{N}$ is a solution of $\eq((\tau P^{-1})A'P,k)/(F,\tau)$,
  where $(\tau P^{-1})A'P=A$.
  Since $f$ is algebraic over $F$, $g$ is a solution of $\eq(A,k)/(F,\tau)$
  algebraic over $F$, which is impossible.
  
  To prove this theorem, we assume that $\eq(A,1)/(C(t),\tau)$ has
  a solution $f\in\mathcal{U}$ differentially algebraic over $F$.
  We have
  \[
    Dr=-\frac{2}{q^3t^3}\neq 0
  \]
  and for all $i\geq 0$,
  \[
    v_\infty(\tau^ir)=v_\infty\left(\frac{1}{q^{2i+3}t^2}\right)=2>0,
  \]
  where $v_\infty$ is the normalized discrete valuation 
  associated with the place $P_\infty$ of $F/K$.
  By Lemma \ref{2017-05-11p1a},
  there exists $g\in F(f)$ such that $\eq(g;P;f)$ holds and
  that it is a solution of $\eq((\tau P)AP^{-1},1)/(F,\tau)$ in $\mathcal{U}$.
  The latter implies that $g\in\mathcal{U}$ is a solution of
  $\eq(B,1)/(F,\tau)$.
  Since $g,g',\ldots$ are elements of the differential field $F(f,f',\dots)$, 
  we find
  \[
    \trdeg F(g,g',\dots)/F\leq \trdeg F(f,f',\dots)/F<\infty,
  \]
  and so $g$ is differentially algebraic over $F$.
  By Theorem \ref{2017-04-28p1a},
  there exists $h\in F=K(t)$ such that
  \[
    q^3\tau^2(r)\tau^3(h)+q^2(\tau(r)+1)\tau^2(h)
    -q(\tau(r)+1)\tau(h)-rh+q\tau(r'/r)=0,
  \]
  from which we have
  \[
    \frac{1}{q^4t^2}\tau^3(h)+\left(\frac{1}{q^3t^2}+q^2\right)\tau^2(h)
    -\left(\frac{1}{q^4t^2}+q\right)\tau(h)-\frac{1}{q^3t^2}h
    -\frac{2}{t}=0.
  \]
  Since $h\in K(t)$ is non-zero, it is expressed as
  \[
    h=\sum_{i=m}^\infty a_i\frac{1}{t^i},\quad a_i\in K,\ a_m\neq 0.
  \]
  Hence we obtain
  \[
  \begin{aligned}
    &\frac{1}{q^4t^2}\sum_{i=m}^\infty\frac{\tau^3a_i}{q^{3i}}\frac{1}{t^i}
    +\left(q^2+\frac{1}{q^3t^2}\right)\sum_{i=m}^\infty\frac{\tau^2a_i}{q^{2i}}\frac{1}{t^i}\\
    &\quad -\left(q+\frac{1}{q^4t^2}\right)\sum_{i=m}^\infty\frac{\tau a_i}{q^i}\frac{1}{t^i}
    -\frac{1}{q^3t^2}\sum_{i=m}^\infty a_i\frac{1}{t^i}-\frac{2}{t}=0.
  \end{aligned}
  \]
  
  In the case $m\neq 1$, by comparing the coefficients of $1/t^m$,
  it follows that
  \[
    q^2\frac{\tau^2a_m}{q^{2m}}-q\frac{\tau a_m}{q^m}=0,
  \]
  from which we have
  \[
    \tau a_m=q^{m-1}a_m.
  \]
  Dividing by $t^{m-1}$, we find
  \[
    \tau\left(\frac{a_m}{t^{m-1}}\right)=\frac{a_m}{t^{m-1}}\in F,
  \]
  which implies $a_m/t^{m-1}\in C\subset K$, a contradiction.
  
  In the case $m=1$, by comparing the coefficients of $1/t^m$, we obtain
  \[
    q^2\frac{\tau^2a_m}{q^{2m}}-q\frac{\tau a_m}{q^m}-2=0,
  \]
  and so
  \begin{equation}\label{2017-05-20p10a}
    \tau a_1=a_1+2.
  \end{equation}
  Differentiating it, we find
  \[
    q\tau D(a_1)=D\tau a_1=D(a_1)\in F,
  \]
  and multiplying it by $t$,
  \[
    \tau(D(a_1)t)=D(a_1)t\in F.
  \]
  Hence $D(a_1)t\in C\subset K$, which implies $D(a_1)=0$.
  Since $a_1\in F$, we conclude $a_1\in C$, which contradicts
  the above equation \eqref{2017-05-20p10a}.
  We found a contradiction in any case,
  which proves this theorem.
\end{proof}

\begin{corollary}\label{2017-07-17p1a}
  Suppose that $q$ is not a root of unity.
  Let $\mathcal{U}=(U,D,\tau)$ be a DTC overfield of $\mathcal{L}$ with
  $D\tau=q\tau D$, and
  $\mathcal{F}$ a DTC intermediate field of $\mathcal{U}/\mathcal{L}$ such that
  \[
    \{x\in F\;|\; Dx=0\}=\{x\in F\;|\;\tau x=x\}=C
  \]
  and $\mathcal{F}=\mathcal{L}\gen{f_1,\dots,f_n}_\mathcal{U}$, where 
  $f_1,\dots,f_n$ satisfy
  \[
    \tau f_i=\alpha_if_i+\beta_i,\quad 
    \alpha_i,\beta_i\in C(t),
    \ \alpha_i\neq 0.
  \]
  Then $\eq(A,1)/(C(t),\tau)$ has no solution differentially algebraic 
  over $F$ in $\mathcal{U}$.
\end{corollary}

\begin{proof}
  We may suppose that $S=\{D^jf_i\;|\;1\leq i\leq m,\,j\geq 0\}$
  are algebraically independent over $C(t)$ and that
  $f_i,Df_i,\dots$ are algebraically dependent
  over $C(t)(S)$ for all $i\geq m+1$.
  Let \[\mathcal{F}'=\mathcal{L}\gen{f_1,\dots,f_m}_\mathcal{U}=(F',D,\tau),\]
  where $F'=C(t)(S)$.
  Assume that $\eq(A,1)/(C(t),\tau)$ has a solution $g\in\mathcal{U}$
  differentially algebraic over $F$.
  Since we have $\trdeg F(g,Dg,\dots)/F<\infty$ and $\trdeg F/F'<\infty$
  by 
  $F=C(t)(\{D^jf_i\;|\;1\leq i\leq n,\,j\geq 0\})$,
  we find $\trdeg F(g,Dg,\dots)/F'<\infty$.
  Hence $\trdeg F'(g,Dg,\dots)/F'<\infty$ holds,
  which means that $g$ is differentially algebraic over $F'$ in $\mathcal{U}$.
  However, Theorem \ref{2017-05-20p2a} says that
  $\eq(A,1)/(C(t),\tau)$ has no solution differentially algebraic over $F'$
  in $\mathcal{U}$. Therefore we found a contradiction.
\end{proof}

\begin{acknowledgement}
  This work was partially supported by JSPS KAKENHI Grant Number 26800049.
\end{acknowledgement}

\begin{flushleft}
  Seiji Nishioka\\
%  Department of Mathematical Sciences\\
  Faculty of Science, Yamagata University\\
  1-4-12 Kojirakawa-machi, Yamagata-shi\\
  990-8560, Japan\\
  e-mail: nishioka@sci.kj.yamagata-u.ac.jp
\end{flushleft}


\begin{thebibliography}{99}
\bibitem{Cohn}Cohn, R. M., 
 {\it Difference Algebra}, 
 Interscience Publishers, 
 New York $\cdot$ London $\cdot$ Sydney, 1965.
\bibitem{DHR}Dreyfus, T., Hardouin, C., and Roques, J.,
  {\it  Hypertranscendence of solutions of Mahler equations}.
  arXiv:1507.03361
\bibitem{HS2008} Hardouin, C. and Singer, M. F.,
  {\it Differential Galois theory of linear difference equations},
  Math. Ann., 342 (2008), 333–-377.
\bibitem{Levin2008}Levin, A.,
  {\it Difference Algebra},
  Springer Science+Business Media B.V., 2008.
\bibitem{Sei2016}Nishioka, S., 
  {\it Proof of unsolvability of q-Bessel equation using valuations},
  J. Math. Sci. Univ. Tokyo,  Vol. 23 (2016), No. 4, Page 763--789.
\bibitem{Sei2017}Nishioka, S.,
  {\it Transcendence of solutions of q-Airy equation},
  Josai Mathematical Monographs, Vol. 10 (2017), 129--137.
\bibitem{Tietze1905} Tietze, H.,
  {\it \"Uber Funktionalgleichungen, deren L\"osungen keiner algebraischen Differentialgleichung gen\"ugen k\"onnen.},
  Monatsh. f\"ur Math. 16 (1905), 329--364.
\bibitem{ZS1958} Zariski, O., Samuel, P.,
  Commutative Algebra Volume I,
  Springer-Verlag, New York, NY, 1958.
\end{thebibliography}
\end{document}